\pdfoutput=1
 
\documentclass[conference,fleqn]{IEEEtran}
\usepackage[cmex10]{amsmath}
\usepackage{amssymb}
\usepackage{amsthm}
\usepackage{mathtools}
\usepackage[pdftex]{graphicx}
\usepackage[font=footnotesize]{subfig}
\usepackage{cite}
\usepackage{flushend}

\usepackage{hyperref}
\hypersetup{%
  pdftitle={},
  pdfauthor={Christian Bauckhage},
  pdfsubject={data science, data mining, machine learning, pattern recognition},
  pdfkeywords={matrix factorization, archetypal analysis, SiVM},
  colorlinks=true,
  urlcolor=blue,
  citecolor=green,
  bookmarks=false}
  
\newcommand{\mat}{\boldsymbol}
\renewcommand{\vec}{\boldsymbol}

\newtheorem{lemma}{Lemma}

\graphicspath{{./Figures/}}

\begin{document}

\title{A Note on  Archetypal Analysis and the Approximation of Convex Hulls}

\author{%
\IEEEauthorblockN{Christian Bauckhage}
\IEEEauthorblockA{B-IT, University of Bonn, Germany \\
    Fraunhofer IAIS, Sankt Augustin, Germany}}

\maketitle

\begin{abstract}
We briefly review the basic ideas behind archetypal analysis for matrix factorization and discuss its behavior in approximating the convex hull of a data sample. We then ask how good such approximations can be and consider different cases. Understanding archetypal analysis as the problem of computing a convexity constrained low-rank approximation of the identity matrix provides estimates for archetypal analysis and the SiVM heuristic.
\end{abstract}

\section{Introduction}

Archetypal analysis (AA) is a matrix factorization technique due to Cutler and Breiman \cite{Cutler1994-AA}. Although not as widely applied  as other methods, AA is a powerful tool for data mining, pattern recognition, or classification as it provides significant \textit{and} easily interpretable results in latent component analysis, dimensionality reduction, and clustering. Robust and efficient algorithms have become available \cite{Bauckhage2009-MAA,Eugster2011-WAR,Morup2012-AAF,Seth2014-PAA,Thurau2010-YWC,Thurau2011-CNN} and examples for practical applications are found in physics \cite{Stone1996-AAOSTD,Stone2002-EAD,Chan2003-AAOGS}, genetics and phytomedicine \cite{Huggins2007-TTH,Roemer2012-EDS,Thogersen2013-AAO}, market research and marketing \cite{Li2003-AAA,Bauckhage2014-KAA,Sifa2014-AGR}, performance evaluation \cite{Porzio2006-AAF,Eugster2012-PAO,Seiler2013-AS}, behavior analysis \cite{Thurau2011-IAA,Drachen2012-GSA,Sifa2013-AM}, and computer vision \cite{Marinetti2005-MFM,Thurau2009-AII,Cheema2011-ARB,Prabhakaran2012-AMS,Asbach2013-UMS,Xiong2013-FRV}. 

Here, we briefly review basic characteristics and geometric properties of AA and approach the problem of bounding its accuracy. We revisit the SiVM heuristic \cite{Thurau2010-YWC} for fast approximative AA and compare its best possible performance to that of conventional AA. In order for this note to be as self-contained as possible, we include a glossary of terms and concepts that frequently occur in the context of archetypal analysis.

\section{Archetypal Analysis}

Assume a matrix $\mat{X} = [\vec{x}_1, \vec{x}_2, \ldots, \vec{x}_n] \in \mathbb{R}^{m \times n}$ of data vectors. Archetypal analysis consists in estimating two \textit{column stochastic} matrices  $\mat{B} \in \mathbb{R}^{n \times k}$ and $\mat{A} \in \mathbb{R}^{k \times n}$ such that 
\begin{equation}
\label{eq:AA}
\mat{X} \approx \mat{Z} \mat{A} = \mat{X} \mat{B} \mat{A}
\end{equation}
where the parameter $k$ can be chosen freely by the analyst, but typically $k \ll \min\{m, n\}$. 

The $k$ columns $\vec{z}_j$ of $\mat{Z} \in \mathbb{R}^{m \times k}$ are called the \emph{archetypes} of the data \cite{Cutler1994-AA}. Given \eqref{eq:AA}, we note the following characteristics of archetypal analysis: 

Since $\mat{Z} = \mat{X} \mat{B}$ and $\mat{B}$ is \textit{column stochastic}, we observe that each archetype 
\begin{equation}
\label{eq:bs}
\vec{z}_j = \mat{X} \vec{b}_j
\end{equation}
is a \textit{convex combination} of columns of $\mat{X}$ where the vector of coefficients $\vec{b}_j$ corresponds to the $j$-th column of $\mat{B}$. 

\begin{figure}[t]
\centering
\subfloat[data points $\vec{x}_i \in \mathbb{R}^2$ and their convex hull]{\includegraphics[width=0.32\columnwidth]{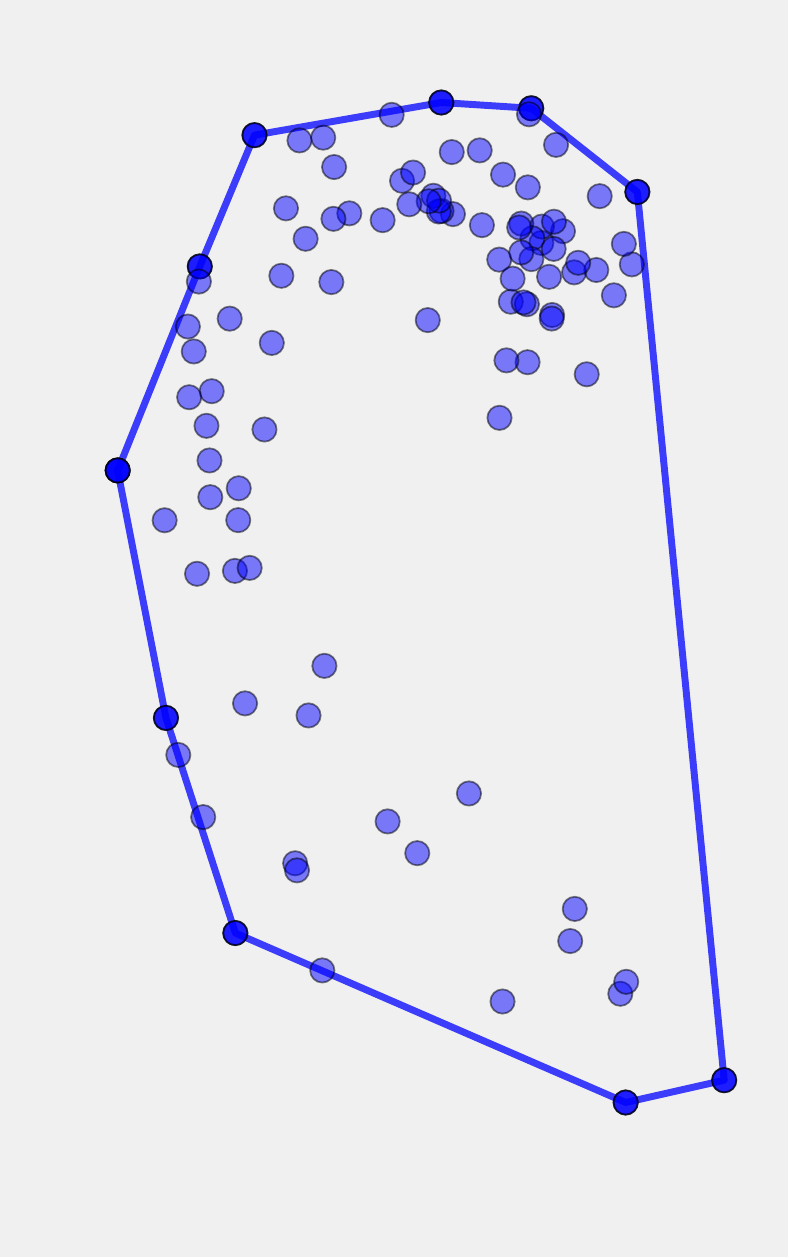}} \hfill
\subfloat[$k=3$ archetypes $\vec{z}_j$ and archetypal hull]{\includegraphics[width=0.32\columnwidth]{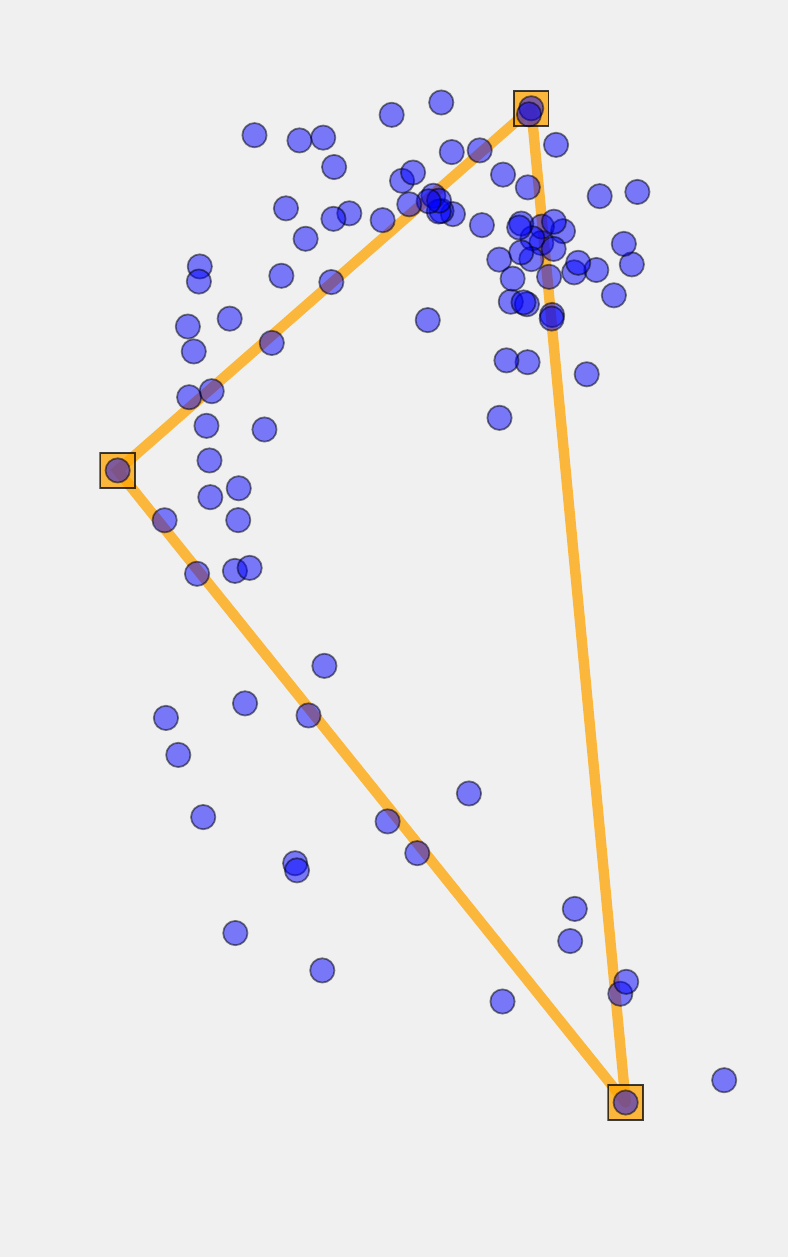}} \hfill
\subfloat[$5=3$ archetypes $\vec{z}_j$ and archetypal hull]{\includegraphics[width=0.32\columnwidth]{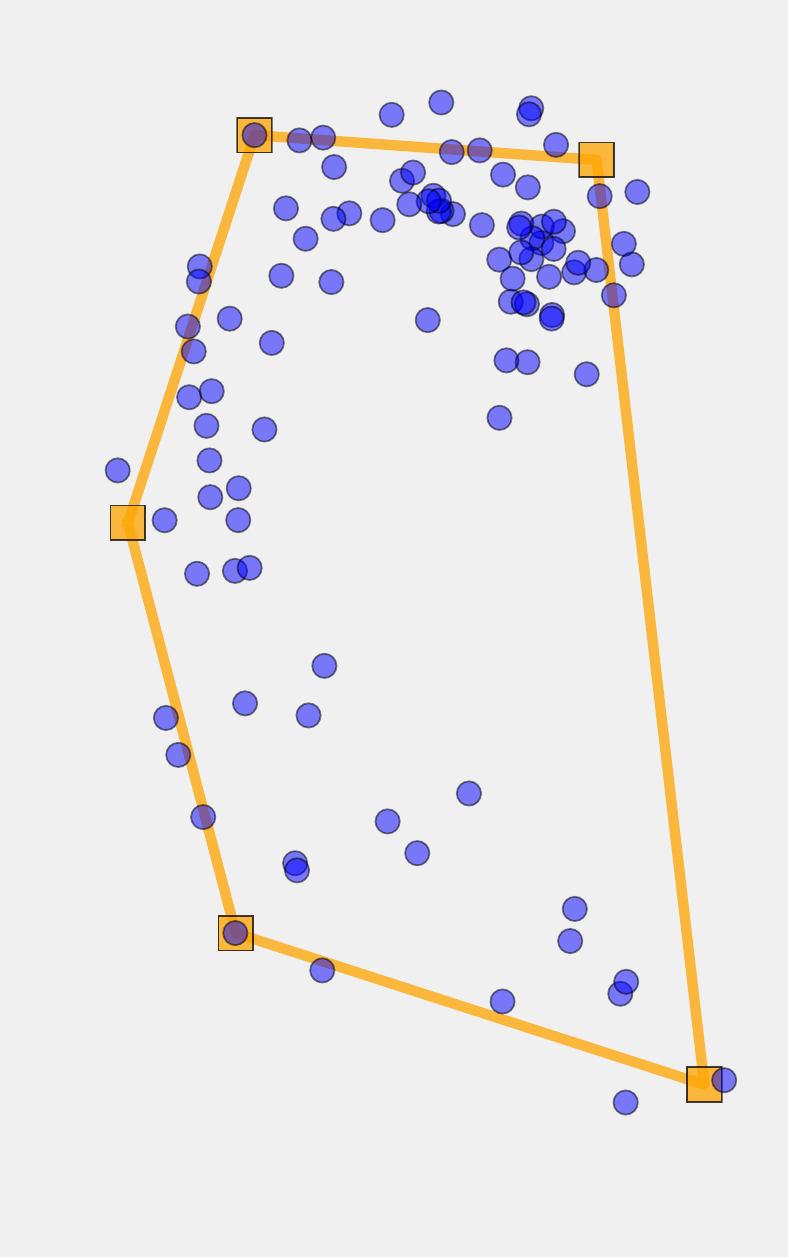}}
 
\subfloat[$7$ archetypes]{\includegraphics[width=0.32\columnwidth]{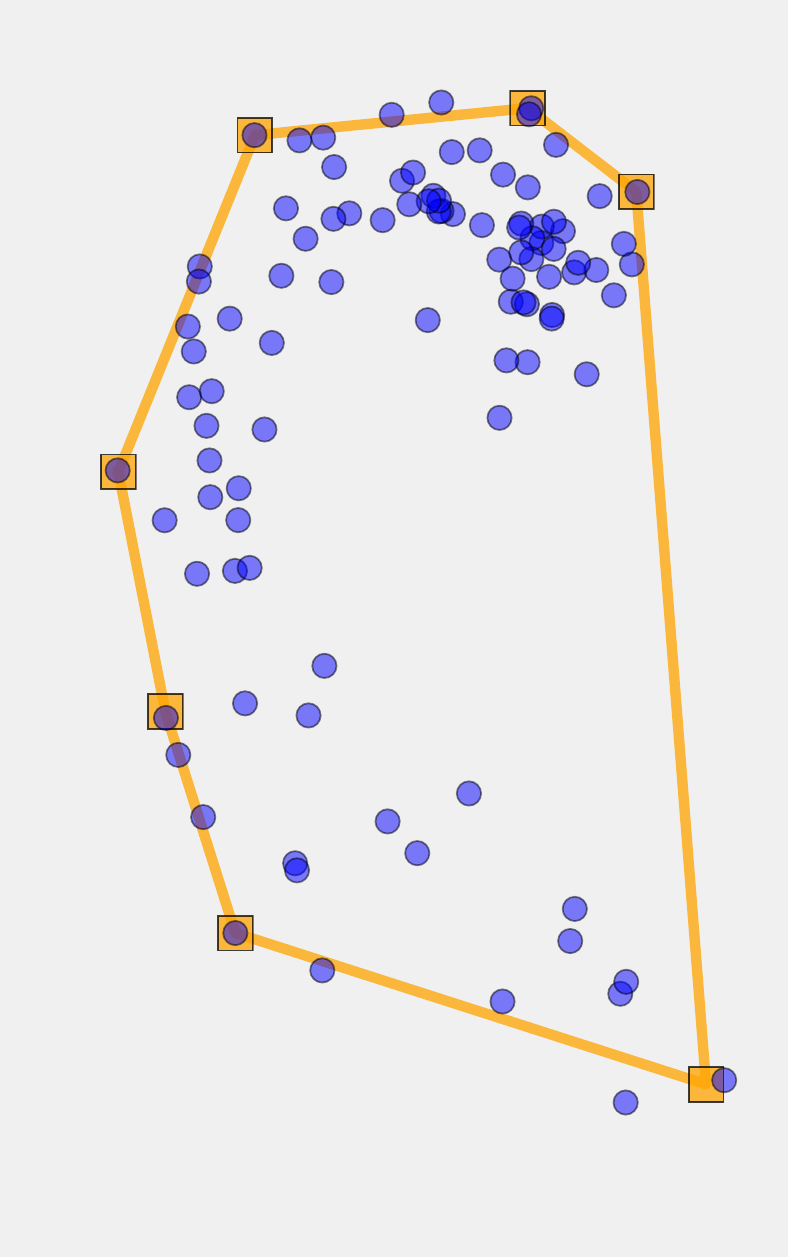}} \hfill
\subfloat[$9$ archetypes]{\includegraphics[width=0.32\columnwidth]{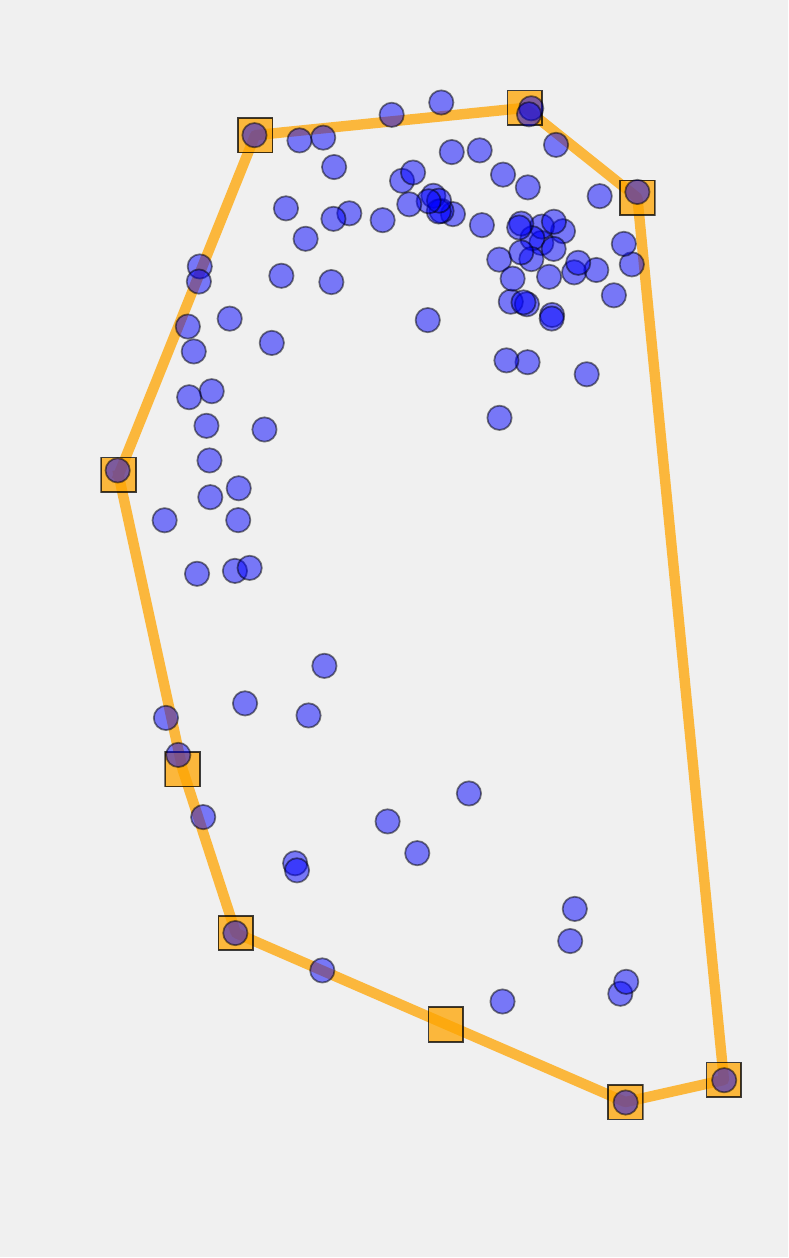}} \hfill 
\subfloat[$10$ archetypes]{\includegraphics[width=0.32\columnwidth]{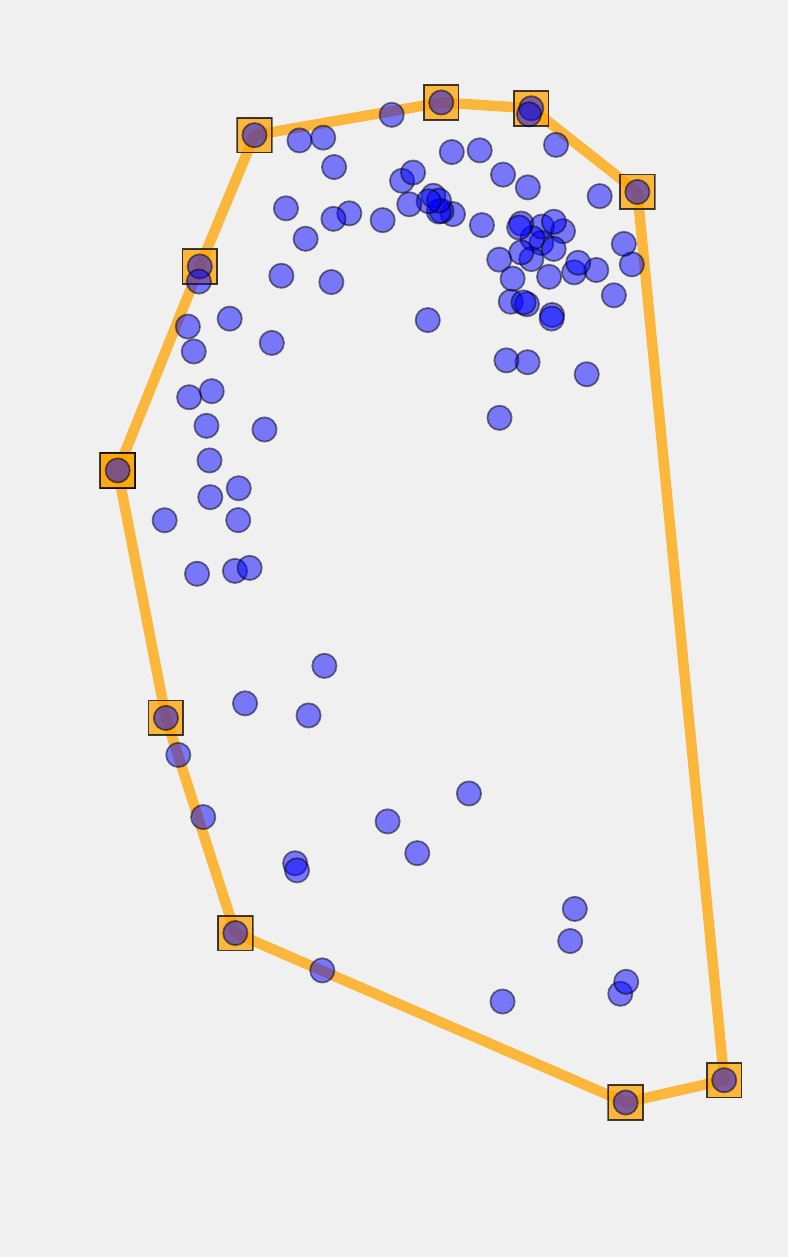}} 
\caption{Archetypal hulls approximate the convex hull of a finite set of data; the more archetypes, the better the approximation.}
\label{fig:hulls}
\end{figure}

Since $\mat{X} \approx \mat{Z} \mat{A}$ and $\mat{A}$ is \textit{column stochastic}, we also see that each data vector
\begin{equation}
\label{eq:as}
\vec{x}_i \approx \mat{Z} \vec{a}_i
\end{equation}
is approximated as a \textit{convex combination} of the columns in $\mat{Z}$ where the vector of coefficients $\vec{a}_i$ corresponds to the $i$-th column of $\mat{A}$. 

Archetypal analysis therefore introduces symmetry into the idea of matrix factorization: \emph{archetypes are convex combinations of data points and data points are approximated in terms of convex combinations of archetypes}.

Results obtained from archetypal analysis are therefore intuitive and physically plausible. In particular, they will be faithful to the nature of the data. For example, if the given data are all non-negative, archetypes determined from archetypal analysis will be non-negative, too. Also, since archetypes are convex combinations of actual data, they closely resemble certain data points and are therefore easily interpretable. 

The problem of computing archetypal analysis can be cast as the following constrained quadratic optimization problem
\begin{align}
\label{eq:E}
\min_{\mat{A}, \mat{B}} \; E = \; &  \bigl \lVert \mat{X} - \mat{X} \mat{B} \mat{A} \bigr \rVert^2 \\
\text{subject to} \quad & \vec{b}_j \succeq \mat{0} \wedge \vec{1}^T \vec{b}_j = 1 \; \forall \; j \in \{1, \ldots, k\}  \notag \\
                        & \vec{a}_i \succeq \mat{0} \wedge \vec{1}^T \vec{a}_i = 1 \; \forall \,\; i \in \{1, \ldots, n\} \notag 
\end{align}
where $\lVert \cdot \rVert$ denotes the matrix Frobenius norm.

The minimization problem in \eqref{eq:E} resembles the objectives of related paradigms and therefore poses difficulties known from other matrix factorization approaches. In particular, we note that, although the objective function $E$ is convex in either $\mat{A}$ or $\mat{B}$, it is anything but convex in the product $\mat{A} \mat{B}$ and may therefore have numerous local minima. Hence, archetypal analysis is a non-trivial problem in general.

Next, we briefly discuss geometric properties of archetypal analysis and how they allow for efficient solutions. However, algorithms that solve \eqref{eq:E} are not our focus in this paper but can be found in \cite{Cutler1994-AA,Bauckhage2009-MAA,Eugster2011-WAR,Morup2012-AAF,Seth2014-PAA}.

\section{The Geometry of AA}

The astute reader may have noticed that archetypal analysis could be solved \textit{perfectly}, if $k=n$. In this case, both factor matrices could be set to the $n \times n$ identity matrix $\mat{I}$ which would comply with the constraints in \eqref{eq:E} and trivially result in $\lVert \mat{X} - \mat{X} \mat{B} \mat{A} \rVert = 0$. 

A more interesting result is that, in addition to the trivial solution, AA can achieve perfect reconstructions even in cases where $k < n$.

Cutler and Breiman \cite{Cutler1994-AA} prove that, for $k > 1$, archetypes necessarily reside on the data convex hull. Moreover, if the hull has $q \leq n$ vertices and we choose $k = q$, these vertices are the unique global minimizers of \eqref{eq:E}. Increasing the number $k$ of archetypes towards the number $q$ of vertices therefore improves the approximation of the convex hull (see Fig.~\ref{fig:hulls}).

These observations allow for a simplification of the problem in \eqref{eq:E}. Assume the vertices of the data convex hull were known and collected in $\mat{V} = [\vec{v}_1, \vec{v}_2, \ldots, \vec{v}_q] \in \mathbb{R}^{m \times q}$. Archetypes could then be determined solely from matrix $\mat{V}$ by solving the appropriately constrained problem
\begin{equation}
\label{eq:V}
\min_{\mat{A}_q, \mat{B}_q} \;  \bigl \lVert \mat{V} - \mat{V} \mat{B}_q \mat{A}_q \bigr \rVert^2 \\
\end{equation}
where $\mat{B}_q$ is of size $q \times k$ and $\mat{A}_q$ of size $k \times q$. Once  archetypes $\mat{Z} = \mat{V} \mat{B}_q$ have been determined,  coefficients $\vec{a}_i$ for those data points not contained in $\mat{V}$ can be computed in a second step. The appeal of this idea lies in the fact that typically $q \ll n$ so that solving \eqref{eq:V} is less demanding than solving \eqref{eq:E}. Moreover, once $\mat{Z}$ is available,  the problem in \eqref{eq:E} needs only to be solved for matrix $\mat{A}$ which requires less effort than estimating $\mat{B}$ and $\mat{A}$ simultaneously \cite{Bauckhage2009-MAA}.

Alas, there is no free lunch! This approach requires to identify the vertices of the data convex hull which is expensive if the data are high dimensional \cite{DeBerg2000-CG}. Nevertheless, there are efficient heuristics for sampling the convex hull of a finite set of data \cite{Bauckhage2009-MAA,Thurau2010-YWC,Kersting2012-MFA,Winter1999-NFA,Chang2006-ANG,Miao2007-EEF} and approximative archetypal analysis can indeed be computed for very large data matrices. 

In this paper, we consider \eqref{eq:V} as a vantage point for reasoning about the quality of archetypal analysis and related heuristics. Accordingly, we henceforth tacitly assume that the $m \times q$ matrix $\mat{V}$ containing the vertices of the data in $\mat{X}$ was available. Moreover, to reduce clutter, we will drop the subscript $q$ of $\mat{B}_q$ and $\mat{A}_q$ and write both matrices as $\mat{B}$ and $\mat{A}$, respectively.

\section{Quality Estimates for AA}

Looking at \eqref{eq:V} we observe that the squared Frobenius norm may be rewritten as 
\begin{equation}
\label{eq:H}
\bigl \lVert \mat{V} - \mat{V} \mat{B} \mat{A} \bigr \rVert^2 
= \bigl \lVert \mat{V} \bigl( \mat{I} - \mat{B} \mat{A} \bigr) \bigr \rVert^2
%= \bigl \lVert \mat{V} \bigl( \mat{I} - \mat{H} \bigr) \bigr \rVert^2
\end{equation}
where the identity matrix $\mat{I}$ as well as the product matrix $\mat{BA}$ are both of size $q \times q$. However, while $\mat{I}$ is of full rank $q$, the rank of $\mat{BA}$ cannot exceed $k$, because its columns are convex combinations of the $k$ columns of $\mat{B}$. 

Hence, given \eqref{eq:H} and assuming that $k < q$, we can interpret archetypal analysis as the problem of finding a convexity constrained low-rank approximation of $\mat{I}$. In this section, we explore to what extent this insight allows bounding the quality of archetypal analysis and approximative heuristics. In particular, we note that  
\begin{equation}
\label{eq:bound}
\bigl \lVert \mat{V} \bigl(\mat{I} - \mat{BA} \bigr) \bigr \rVert^2 \leq \bigl \lVert \mat{V} \bigr \rVert^2 \, \bigl \lVert \mat{I} - \mat{BA} \bigr \rVert^2
\end{equation}
and ask how large or small can $\bigl \lVert \mat{I} - \mat{BA} \bigr \rVert^2$ become for $k < q$?

\subsection{Two general observations}

To begin with, we show that a perfect low-rank approximation of $\mat{I}$ is impossible. For the proof, we resort to geometric properties of the standard simplex to familiarize ourselves with arguments used later on.

\begin{lemma}
Let $\mat{I}$ be the $q \times q$ identity matrix and $\mat{B} \in \mathbb{R}^{q \times k}$ and $\mat{A} \in \mathbb{R}^{k \times q}$ be two column stochastic matrices with $k < q$. Then $\bigl \lVert \mat{I} - \mat{BA} \bigr \rVert = 0$ cannot be achieved.
\end{lemma}

\begin{proof}
Note that the columns $\vec{b}_j$ of $\mat{B}$ lie the standard simplex $\Delta^{q-1}$ whose vertices correspond to the columns $\vec{e}_i$ of $\mat{I}$. In order for the difference of $\mat{I}$ and $\mat{BA}$ to vanish, their corresponding columns must be equal
\begin{equation}
\label{eq:proof1}
\vec{e}_i = \mat{B} \vec{a}_i = \sum_{j=1}^k \vec{b}_j a_{ji} \;\; \forall \; i \in \{1, \ldots, q\}.
\end{equation}
However, as a vertex of a convex set, $\vec{e}_i$ is \textit{not} a convex combination of other points in that set. Hence, in order for \eqref{eq:proof1} to hold, every $\vec{e}_i$ must be contained in the column set of $\mat{B}$. Yet, while there are $q$ columns in $\mat{I}$, there are only $k < q$ columns in $\mat{B}$.
\end{proof}
 
Having established that a perfect approximation of the unit matrix cannot be attained under the conditions studied here, we ask for the worst case. How bad can it possibly be? 

\begin{lemma}
Let $\mat{I}$, $\mat{B}$, and $\mat{A}$ be given as above. The squared distance between $\mat{I}$ and $\mat{BA}$ has an upper bound of
\begin{equation}
\label{eq:bound1}
\bigl \lVert \mat{I} - \mat{BA} \bigr \rVert^2 \leq 2 q.
\end{equation}
\end{lemma}
\begin{proof}
The columns $\vec{e}_i$ of $\mat{I}$ and the columns $\mat{B} \vec{a}_i$ of $\mat{BA}$ reside in the standard simplex $\Delta^{q-1}$ which is a closed convex set. Any two vertices of $\Delta^{q-1}$ are at a distance of $\sqrt{2}$ and no two points in $\Delta^{q-1}$ can be farther apart than two vertices. Therefore,
\begin{equation}
\bigl \lVert \mat{I} - \mat{BA} \bigr \rVert^2 = \sum_{i=1}^q \lVert \vec{e}_i - \mat{B} \vec{a}_i \rVert^2 \leq q \left(\sqrt{2}\,\right)^2 = 2q.
\end{equation}
\end{proof}

This simple and straightforward result has consequences for archetypal analysis in general. Plugging it into \eqref{eq:bound} gives
\begin{equation}
\bigl \lVert \mat{V} \bigl(\mat{I} - \mat{B} \mat{A} \bigr) \bigr \rVert^2 \leq 2 q \, \bigl \lVert \mat{V} \bigr \rVert^2 
\end{equation}
and we realize that \textit{the number of vertices $q$ of the data convex hull may impact the quality of archetypal analysis}. The more extremes there are, the higher the potential to err. However, the fewer extreme elements a data set contains, the better the chances of archetypal analysis to produce a good solution.

\subsection{Two more specific observations}

Curiously, the upper bound in \eqref{eq:bound1} does not depend on the number $k$ of archetypes to be determined. Our next questions are therefore if it can be tightened and the parameter $k$ be incorporated.

First of all, we note that stochastic rank reduced approximations of the identity matrix can be understood geometrically. Recall that the columns $\mat{B} \vec{a}_i$ of $\mat{BA}$ are convex combinations of the columns $\vec{b}_j$ of $\mat{B}$ which are $q$-dimensional stochastic vectors that reside in $\Delta^{q-1}$. Approximating $\mat{I}$ in terms of $\mat{BA}$ is therefore tantamount to placing $k$ vectors $\vec{b}_j$ into $\Delta^{q-1}$ such that its vertices $\vec{e}_i$ can be approximated in terms of convex combinations over the $\vec{b}_j$.

Also, a good solution to \eqref{eq:V} would mean $\mat{V} \approx \mat{V} \mat{BA}$ so that $\vec{v}_i \approx \mat{V} \mat{B} \vec{a}_i$. If one of the data vertices was reconstructed exactly, that is if $\vec{v}_i = \mat{V} \mat{B} \vec{a}_i$, then $\vec{v}_i = \mat{V} \vec{e}_i$ therefore $\mat{B} \vec{a}_i = \vec{e}_i$. However, $\vec{e}_i$ is a vertex of $\Delta^{q-1}$ and therefore \textit{not} a convex combination of points in $\Delta^{q-1}$. Therefore one of the columns $\vec{b}_j$ of $\mat{B}$ must equal $\vec{e}_i$ and the corresponding coefficient vector  $\vec{a}_i = [0_1, \ldots, 0_{j-1}, 1_j, 0_{j+1}, \ldots, 0_k]^T$

In other words, selecting one of the $\vec{v}_i$ of $\mat{V}$ as an archetype is to place one of the $\vec{b}_j$ of $\mat{B}$ onto a vertex of $\Delta^{q-1}$.

Having established this, we can now analyze the SiVM heuristic for archetypal analysis. SiVM (simplex volume maximization) as introduced in \cite{Thurau2010-YWC} is a greedy search algorithm that determines $k$ data points in $\mat{X}$ that are as far apart as possible and may therefore act as archetypes. In light of our discussion, this heuristic can thus be understood as implicitly placing the $k$ vectors $\vec{b}_j$ onto $k$ of the $q$ vertices of the standard simplex $\Delta^{q-1}$. 

\begin{figure}[t]
\centering
\includegraphics[width=0.6\columnwidth]{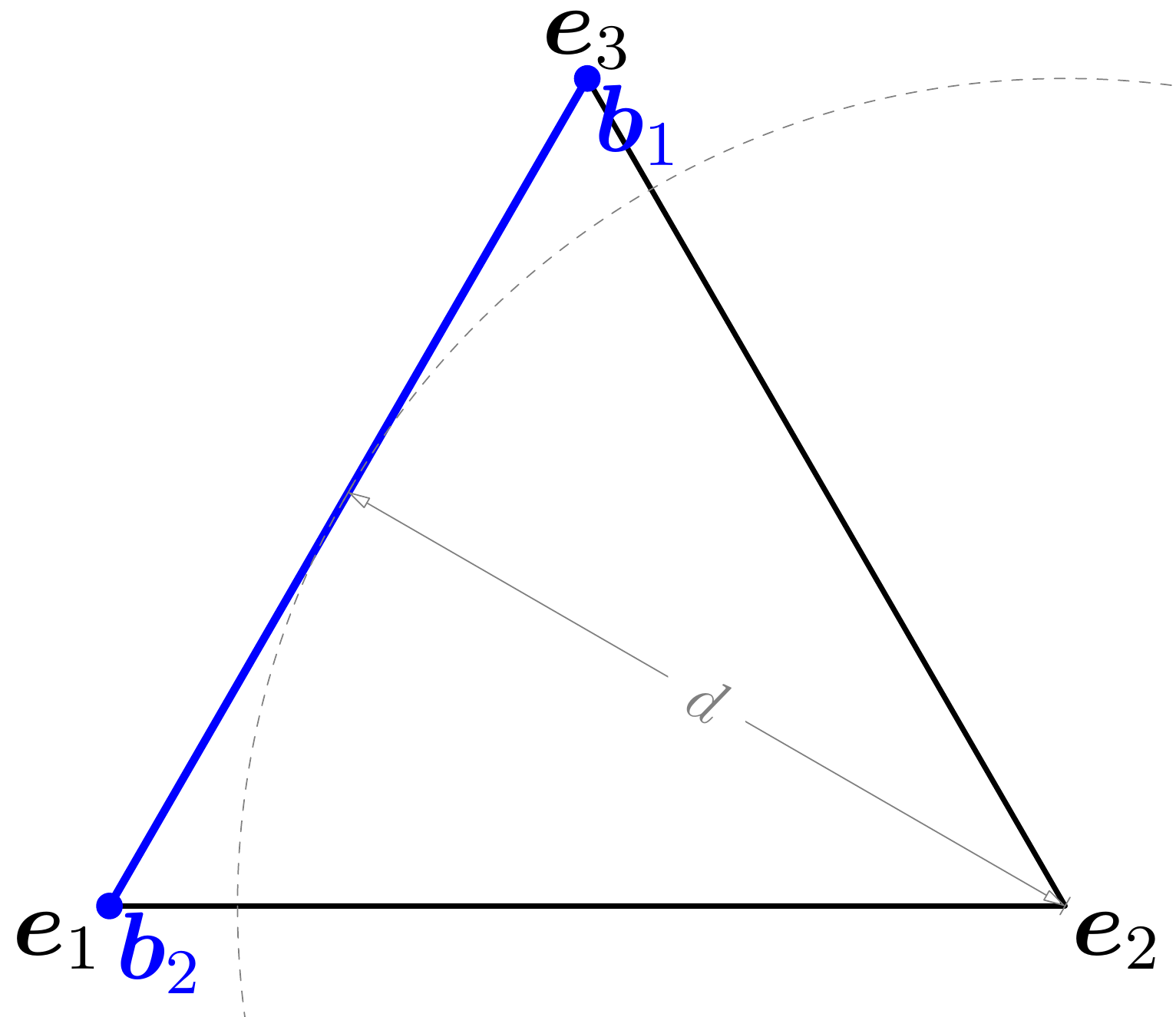}
\caption{\label{fig:height} If two stochastic vectors $\vec{b}_1$ and $\vec{b}_2$ are placed onto two of the vertices of the standard simplex $\Delta^2$, the distance of the third vertex to the sub-simplex spanned by the $\vec{b}_j$ corresponds to the height of $\Delta^2$.}
\end{figure}

Say $\vec{e}_1, \ldots, \vec{e}_k$ had been chosen this way. Then what is the minimal reconstruction error for any of the remaining $q-k$ vertices $\vec{e}_l$? Again, we can argue geometrically. As the selected vertices form a subsimplex $\Delta^{k-1}$ of $\Delta^{q-1}$, any remaining vertex $\vec{e}_l$ would ideally be approximated by its closest point in $\Delta^{k-1}$. The distance $d$ between $\vec{e}_l$ and $\Delta^{k-1}$ corresponds to the height of the $k$ simplex formed by $\{\vec{e}_1, \ldots, \vec{e}_k, \vec{e}_l\}$ (see Fig.~\ref{fig:height}). Given the expression for the height of a standard simplex, we therefore have
\begin{equation}
d = \sqrt{\frac{k+1}{2k}} \cdot \sqrt{2}.
\end{equation}

This then establishes that, running SiVM for $k \leq q$, we can achieve
\begin{align}
\label{eq:bound2}
\bigl \lVert \mat{I} - \mat{BA} \bigr \rVert^2 
& = \underbrace{\sum_{i=1}^k \lVert \vec{e}_i - \mat{B} \vec{a}_i \rVert^2}_{= 0} 
  + \underbrace{\sum_{i=k+1}^q \lVert \vec{e}_i - \mat{B} \vec{a}_i \rVert^2}_{\neq 0} \notag \\
& = (q-k) \cdot d^2  \notag \\
& = (q-k) \cdot \frac{k+1}{k}.
\end{align}

In contrast to the quality estimate in \eqref{eq:bound1}, the one in \eqref{eq:bound2} now depend on $k$. It decreases for growing $k$ and, as expected, would drop to zero, if $k = q$.

As a greedy selection mechanisms, SiVM is fast but suboptimal with respect to reconstruction accuracy. To see this, we next investigate what happens if the $\vec{b}_j$ are not necessarily placed onto vertices $\vec{e}_i$ of the standard simplex.

\begin{figure*}[t]
\centering
\subfloat[$q=3$, $k=2$]{\includegraphics[height=4cm]{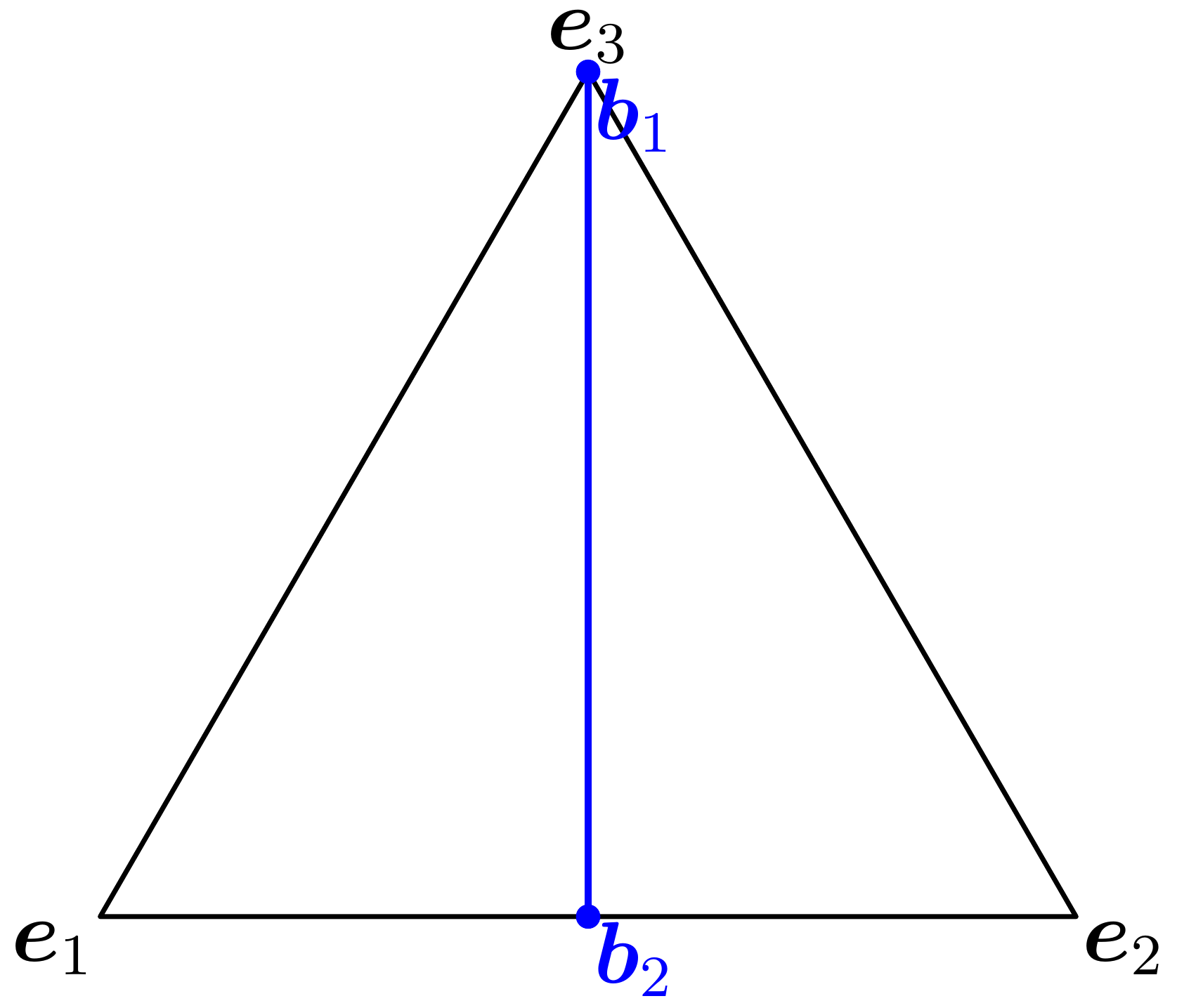}}  \hfill
\subfloat[$q=4$, $k=2$]{\includegraphics[height=4cm]{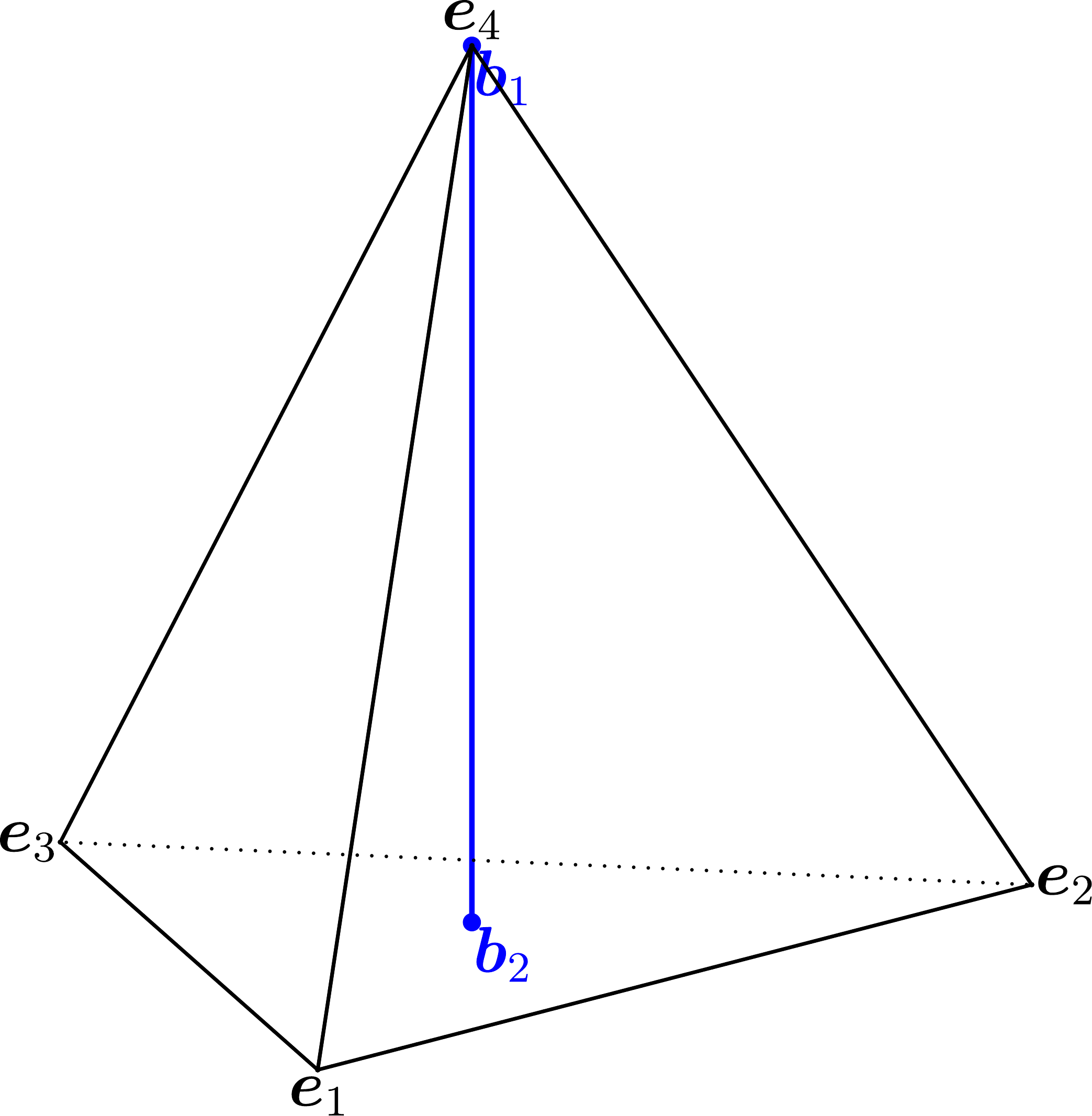}}  \hfill
\subfloat[$q=4$, $k=2$]{\includegraphics[height=4cm]{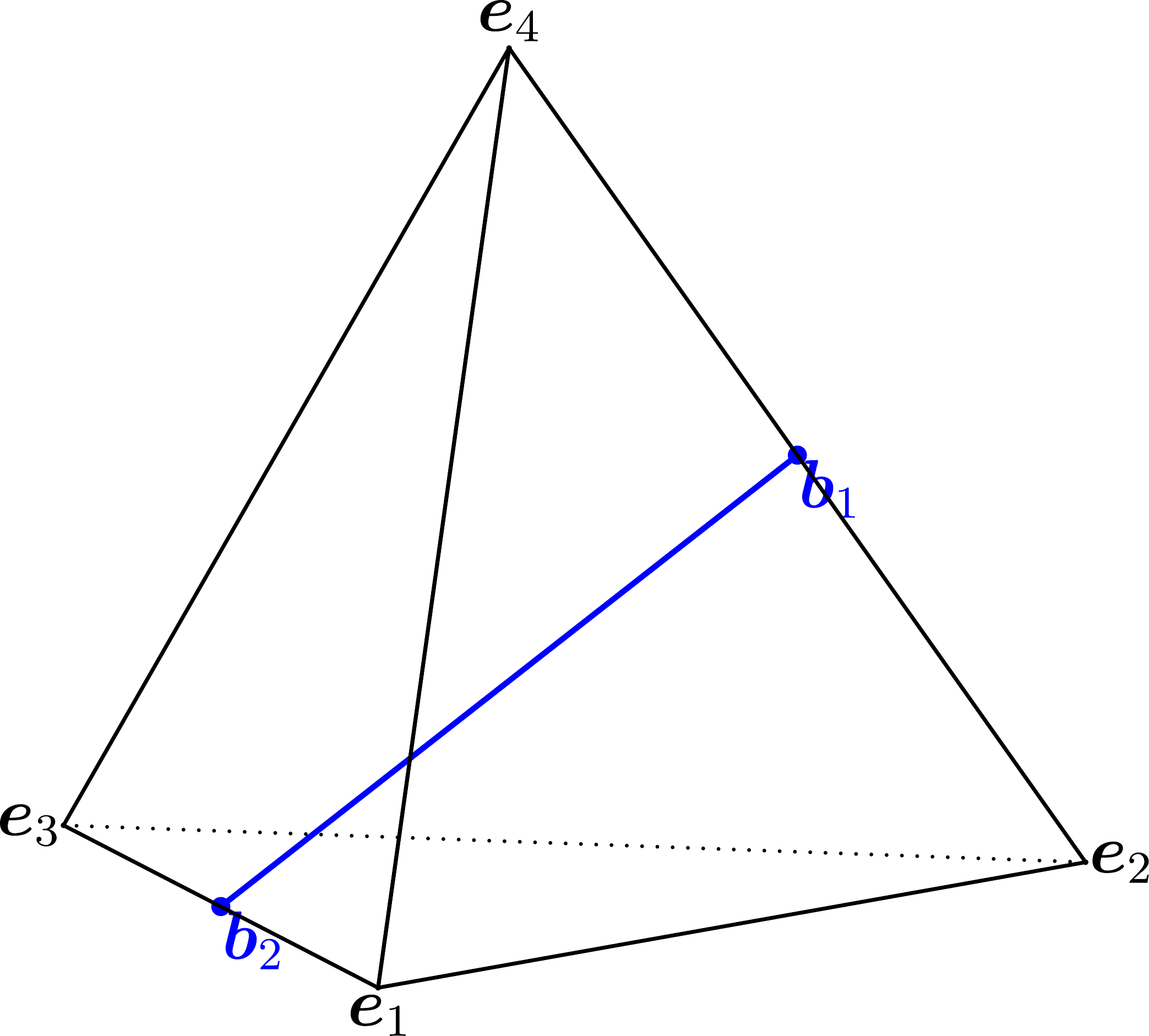}} \hfill
\subfloat[$q=4$, $k=3$]{\includegraphics[height=4cm]{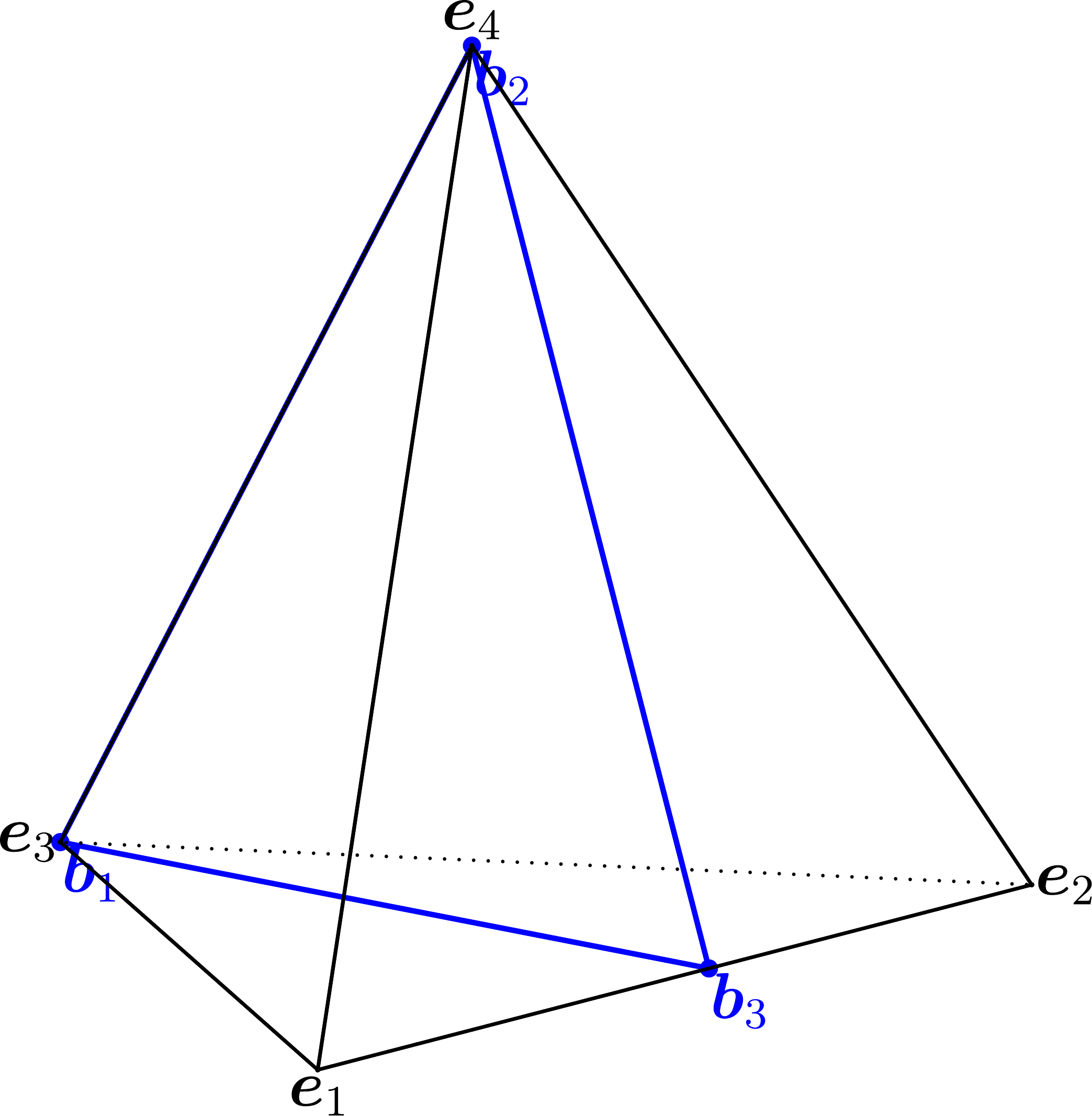}}
\caption{\label{fig:examples} Examples of how to place stochastic vector $\vec{b}_1, \ldots, \vec{b}_k$ into the standard simplex $\Delta^{q-1}$, $k < q$, so that its vertices $\vec{e}_1, \ldots, \vec{e}_q$ can be well approximated in terms of the $\vec{b}_j$. Each of these configurations achieves the accuracy derived in \eqref{eq:bound2}.}
\end{figure*}

We begin by looking at the problem of computing a rank $k=1$ approximation of $\mat{I}$, that is, the problem of placing only a single vector $\vec{b}$ into the standard simplex. Ideally, this vector would have a minimum distance to all the vertices and should therefore minimize
\begin{equation}
\label{eq:b}
E = \sum_{i=1}^q \lVert \vec{e}_i - \vec{b} \rVert^2 = \sum_{i=1}^q \bigl( \vec{e}_i^T \vec{e}_i - 2 \vec{e}_i^T \vec{b} + \vec{b}^T \vec{b} \bigr)
\end{equation}
Deriving this objective function with respect to $\vec{b}$ and equating the result to zero yields
\begin{equation}
2q \vec{b} - 2 \sum_{i=1}^q \vec{e}_i = 0 \; \Leftrightarrow \; \vec{b} = \frac{1}{q} \sum_{i=1}^q \vec{e}_i = \frac{1}{q} \vec{1}.
\end{equation}

Apparently, for $k = 1$, the optimal vector $\vec{b}$ is (of course) the center point of $\Delta^{q-1}$. At the same time, the $k \times q$ matrix $\mat{A}$ of coefficients degenerates to a row vector which, as it must be stochastic, is the vector $\vec{1}^T$. We therefore have
\begin{equation}
\bigl \lVert \mat{I} - \tfrac{1}{q} \vec{1} \vec{1}^T \bigr \rVert^2 = q \cdot 1^2 - q^2 \cdot \left(\frac{1}{q}\right)^2 = q-1.
\end{equation}

Next, we consider the problem of placing $1 < k \leq q$ vectors $\vec{b}_j$ into the standard simplex so as to approximate $\mat{I}$. Given our discussion so far, this problem can be approached in terms of subdividing the standard simplex $\Delta^{q-1}$ into $k$ disjoint sub-simplices and placing one of the $\vec{b}_j$ at the center of each of the resulting simplices. Put in simple terms, if we consider
\begin{equation}
q = \sum_{i=1}^k q_i 
\end{equation}
such that $q_i \in \mathbb{N}$ and $q_i \geq 1$, the identity matrix can indeed be approximated with a reconstruction accuracy of 
\begin{equation}
\label{eq:bound2}
\bigl \lVert \mat{I} - \mat{BA} \bigr \rVert^2 = \sum_{i=1}^k (q_i - 1) = q - k.
\end{equation}

Interestingly, the quality estimate in \eqref{eq:bound2} decreases linearly in $k$. Again, and as expected, it will drop to zero, if $k = q$. The examples in Fig.~\ref{fig:examples} show several configurations for which this reconstruction accuracy is achieved.

Finally, the results in \eqref{eq:bound1} and \eqref{eq:bound2} allow for estimating the relative reconstruction accuracy of the SiVM heuristic in comparison to conventional archetypal analysis. Figure~\ref{fig:quality} plots the ratio
\begin{equation}
\frac{(q-k)}{(q-k) \frac{k+1}{k}} = \frac{k}{k+1}
\end{equation}
as a function of $k$. We observe that, for small choices of $k$, the accuracy of SiVM quickly improves, exceeds $90\%$ for $k > 10$, and then slowly approaches the theoretically optimal performance of conventional AA if $k$ is increased further. 

\begin{figure}
\centering
\includegraphics[width=0.9\columnwidth]{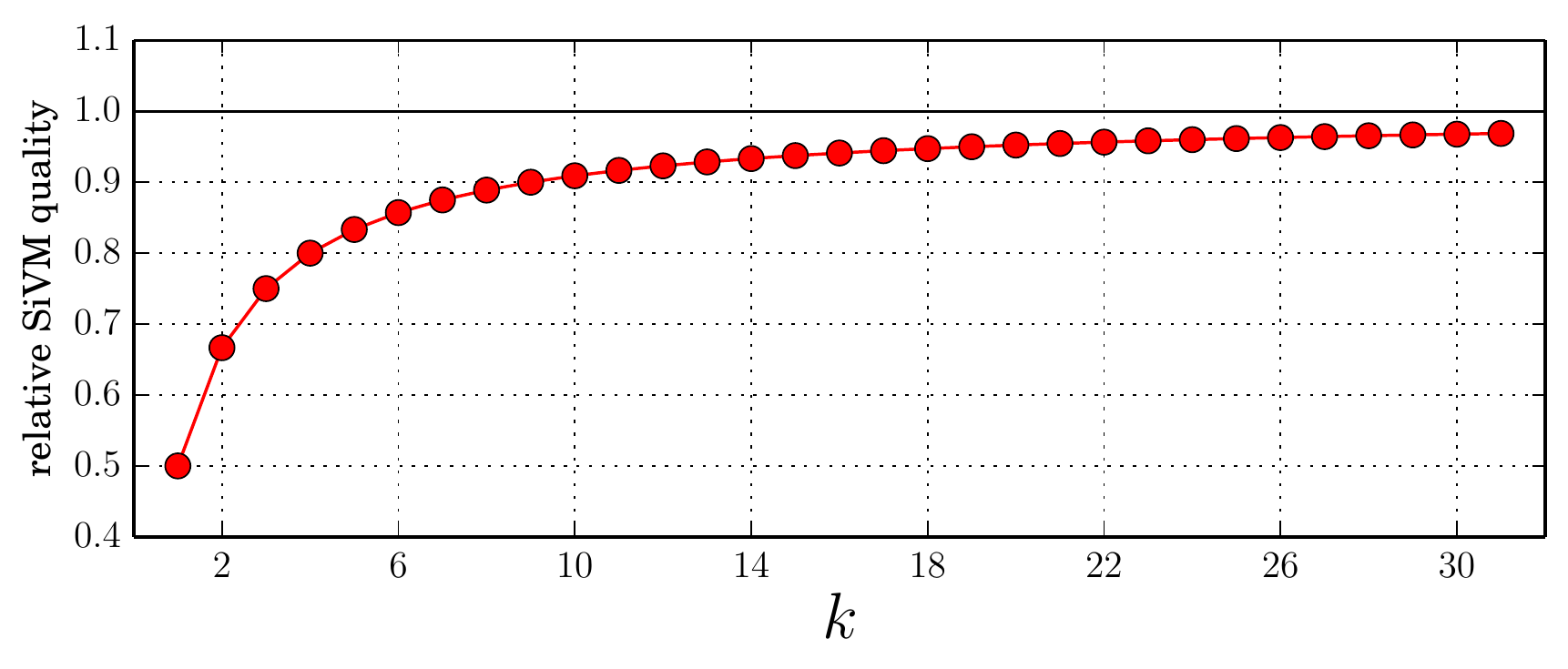}
\caption{\label{fig:quality} Relative reconstruction accuracy of the SiVM heuristic for fast approximative AA. For $k > 10$, can achieve more than $90\%$ of the performance of conventional AA.}
\end{figure}

\section{Summary and Conclusion}

In this note, we briefly reviewed the ideas behind archetypal analysis for data matrix factorization. We discussed geometric properties of solutions found by archetypal analysis and as to how far these allow for efficient heuristics for approximative AA. We then addressed the question of how to estimate the quality of AA algorithms. We showed that the problem of computing archetypal analysis can be reduced to the problem of computing a stochastic low rank approximation of the identity matrix. This point of view allowed us to characterize the optimal performance of SiVM heuristic as well as for conventional AA. Our results indicate that SiVM can provide robust solutions if a larger number of archetypes ($k > 10$) are to be determined.

\section*{Glossary}

Throughout this text, we are concerned with matrices and vectors over the field of real numbers $\mathbb{R}$.

Vectors are written as bold lower case letters ($\vec{v}$) and their components in subscripted lower case italics ($v_k$). $\vec{0}$ is the vector of all zeros and $\vec{1}$ is the vector of all ones.

Matrices are written using bold upper case letters ($\mat{M}$) and doubly subscripted lower case italics ($m_{ij}$) denote individual components. In order to express that a given matrix consists of $n$ column vectors, we also write $\mat{M} = [\vec{m}_1, \vec{m}_2, \ldots, \vec{m}_n]$
where $\vec{m}_j \in \mathbb{R}^m$ is the $j$-th column of $\mat{M} \in \mathbb{R}^{m \times n}$. 

We write $\lVert \mat{M} \rVert$ for the \href{http://en.wikipedia.org/wiki/Matrix_norm#Frobenius_norm}{\emph{Frobenius norm}} of $\mat{M}$ and recall that 
\begin{equation*}
\lVert \mat{M} \rVert^2 = \sum_{i,j} m_{ij}^2 = \operatorname{tr} \bigl( \mat{M}^T \mat{M} \bigr) = \operatorname{tr} \bigl( \mat{M} \mat{M}^T \bigr)
\end{equation*}
as well as
\begin{equation*}
\bigl \lVert \mat{L} \mat{M} \bigr \rVert \leq \bigl \lVert \mat{L} \bigr \rVert \, \bigl \lVert \mat{M} \bigr \rVert.
\end{equation*}

A vector $\vec{v} \in \mathbb{R}^m$ is a \href{http://en.wikipedia.org/wiki/Probability_vector}{\emph{stochastic vector}}, if $v_k \geq 0$ for all $k$ and $\sum_k v_k = 1$. For abbreviation, we write $\vec{v} \succeq \vec{0}$ in order to indicate that $\vec{v}$ is non-negative and use the inner product $\vec{1}^T \vec{v}$ as a convenient shorthand for $\sum_k v_k$.

A matrix $\mat{M} \in \mathbb{R}^{m \times n}$ is \href{http://en.wikipedia.org/wiki/Stochastic_matrix}{\emph{column stochastic}}, if each of its $n$ column vectors $\vec{m}_j$ is stochastic. 

A set $\mathcal{S} \subset \mathbb{R}^m$ is called a \href{http://en.wikipedia.org/wiki/Convex_set}{\emph{convex set}}, if every point on the line segment between any two points in $\mathcal{S}$ is also in $\mathcal{S}$, i.e. if 
\begin{equation*}
\forall \vec{u}, \vec{v} \in \mathcal{S}, \; \forall a \in [0,1]: a \vec{u} + (1-a) \vec{v} \in \mathcal{S}.
\end{equation*}

A vector $\vec{v} \in \mathbb{R}^m$ is called a \href{http://en.wikipedia.org/wiki/Convex_combination}{\emph{convex combination}} of $n$ vectors $\{\vec{v}_1, \vec{v}_2, \ldots, \vec{v}_n\} \subset \mathbb{R}^m$, if 
\begin{equation*}
\vec{v} = \sum_{i=1}^n a_i \vec{v}_i \; \text{ where } \; a_i \geq 0 \; \text{ and } \; \sum_{i=1}^n a_i = 1.
\end{equation*}
Using matrix notation, we write $\vec{v} = \mat{V} \vec{a}$ where the matrix  $\mat{V} = [\vec{v}_1, \vec{v}_2, \ldots, \vec{v}_n] \in \mathbb{R}^{m \times n}$ and  the vector $\vec{a} \in \mathbb{R}^n$ is a stochastic vector.

An \href{http://en.wikipedia.org/wiki/Extreme_point}{\emph{extreme point}} of a convex set $\mathcal{S}$ is any point $\vec{v} \in \mathcal{S}$ that is \textit{not} a convex combination of other points in $\mathcal{S}$. In other words, if $\vec{v}$ is an extreme point and $\vec{v} = a \vec{u} + (1-a) \vec{w}$ for $\vec{u}, \vec{w} \in \mathcal{S}$ and $a \in [0,1]$, then $\vec{v} = \vec{u} = \vec{w}$.

The \href{http://en.wikipedia.org/wiki/Convex_hull}{\emph{convex hull}} $\mathcal{C}(\mathcal{S})$ of a set $\mathcal{S} \subset \mathbb{R}^m$ is the set of all possible convex combinations of points in $\mathcal{S}$, that is
\begin{equation*}
    \mathcal{C}(\mathcal{S}) =
        \biggl \{
        \sum_{\vec{v}_i \in \mathcal{R}} a_i \vec{v}_i 
        \, \biggl| \,
        \mathcal{R} \subseteq \mathcal{S}, \lvert \mathcal{R} \rvert < \infty, \vec{a} \succeq \vec{0}, \vec{1}^T \vec{a} = 1 \biggr \}.
\end{equation*}

A \href{http://en.wikipedia.org/wiki/Convex_polytope}{\emph{convex polytope}} is the convex hull of finitely many points, i.e. it is the set $\mathcal{C}(\mathcal{S})$ for $\lvert \mathcal{S} \rvert < \infty$. The extreme points of a polytope are called \href{http://en.wikipedia.org/wiki/Convex_polytope}{\emph{vertices}}. If $\mathcal{V}(\mathcal{S})$ is the set of all vertices of a polytope, then every point of the polytope can be expressed as a convex combination of the points in $\mathcal{V}(\mathcal{S})$.

\begin{figure}
\centering
\includegraphics[width=0.6\columnwidth]{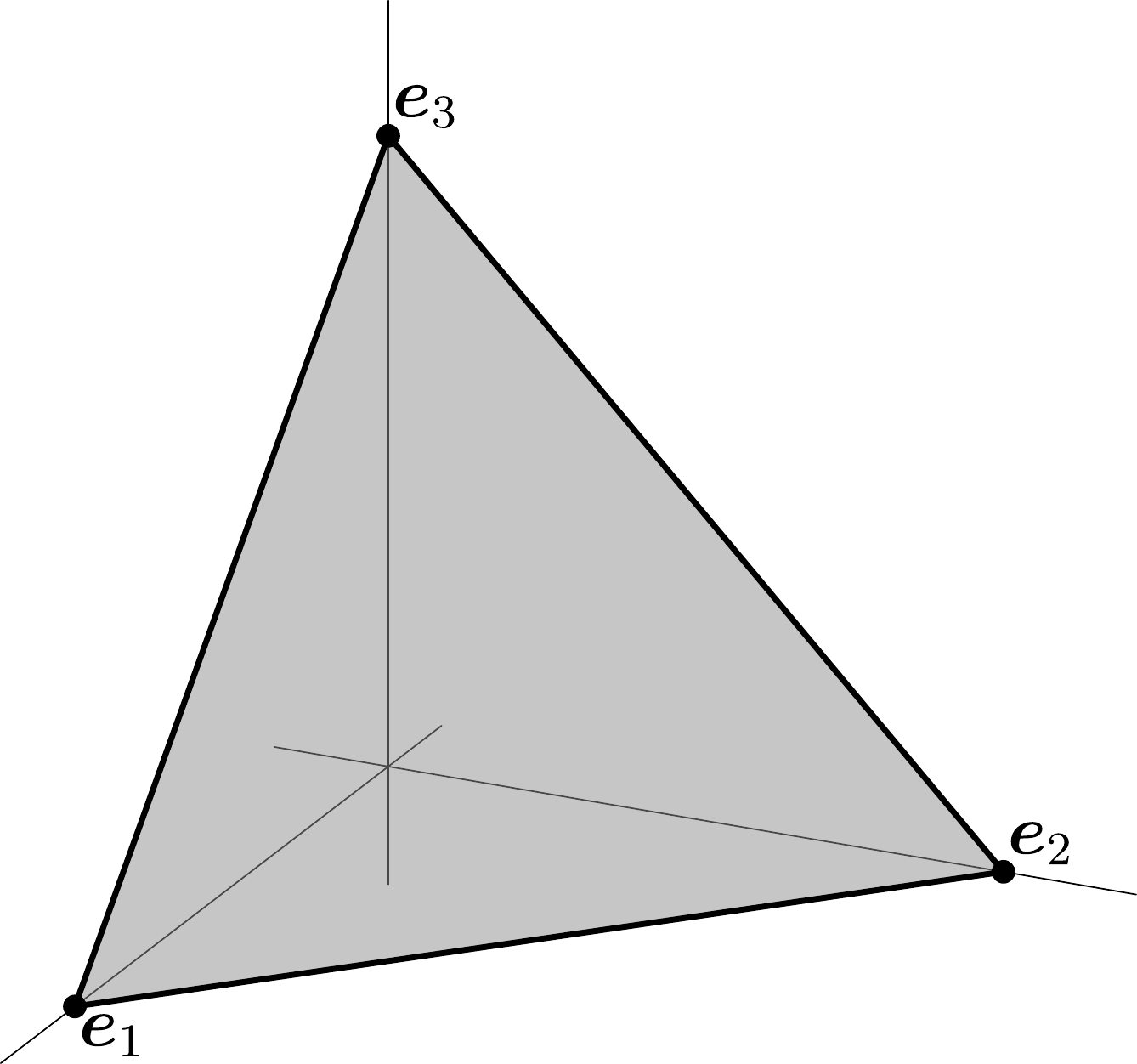}
\caption{\label{fig:2simplex} The standard simplex $\Delta^2$ in $\mathbb{R}^3$.}
\end{figure}

The \href{http://en.wikipedia.org/wiki/Simplex#The_standard_simplex}{\emph{standard simplex}} $\Delta^{m-1}$ is the convex hull of the \href{http://en.wikipedia.org/wiki/Standard_basis}{\emph{standard basis}} $\{ \vec{e}_1, \vec{e}_2, \ldots, \vec{e}_m\} \subset \mathbb{R}^m$, that is
\begin{equation*}
\Delta^{m-1} = \biggl \{ \sum_{i = 1}^m a_i \vec{e}_i \, \biggl| \, \vec{a} \in \mathbb{R}^m, \vec{a} \succeq \vec{0}, \vec{1}^T \vec{a} = 1 \biggr \}.
\end{equation*}
We note that the standard simplex is a polytope whose vertices correspond to the standard basis vectors (see Fig.~\ref{fig:2simplex}) and that every stochastic vector $\vec{v} \in \mathbb{R}^m$ resides within $\Delta^{m-1}$.

There are many useful results about the standard simplex $\Delta^{m-1}$. Among others, its side length is $\sqrt{2}$ and its height is known to be
\begin{equation}
h_{m-1} = \sqrt{\frac{m}{2 (m-1)}} \cdot \sqrt{2}.
\end{equation}

\bibliographystyle{IEEEtran}
\bibliography{literatureMatFact}

\end{document}